\documentclass{amsart}

\usepackage[latin1,utf8]{inputenc}
\usepackage{subfigure}
\usepackage{amssymb,amsmath,amsthm,stmaryrd,comment,enumitem}
\usepackage{bm, enumitem}

\usepackage{graphicx}
\usepackage{hyperref}
\usepackage{epsfig}
\usepackage{color, comment}

\newtheorem{theorem}{Theorem}
\newtheorem{lemma}[theorem]{Lemma}
\newtheorem{proposition}[theorem]{Proposition}
\newtheorem{corollary}[theorem]{Corollary}

\newtheorem{remark}{Remark}

\author{Guillaume Chapuy}
\thanks{
CNRS -- LIAFA, Université Paris Diderot, Paris, France. 
Partial  support from \emph{Agence Nationale de la Recherche}, grant number ANR~12-JS02-001-01 ``Cartaplus'', and from the City of Paris, grant ``\'Emergences Paris 2013, Combinatoire \`a Paris''.
The author is currently visiting McGill University, School of Computer Science, Montr\'eal, Canada, and acknowledges partial support (and hospitality!) from Luc Devroye.
Email:~{\tt Guillaume.Chapuy@liafa.univ-paris-diderot.fr}.
}

\title[The asymptotic number of $12..d$ avoiding words of content $1^r2^r\dots n^r$]{
The asymptotic number of $12..d$-Avoiding Words with $r$ occurrences of each letter  $1,2, ..., n$.
}

\begin{document}
\begin{abstract}
Following Ekhad and Zeilberger (The Personal Journal of Shalosh B. Ekhad and Doron Zeilberger, Dec 5 2014; see also arXiv:1412.2035), we study the asymptotics for large $n$ of the number $A_{d,r}(n)$ of words of length $rn$ having $r$ letters $i$ for $i=1..n$, and having no increasing subsequence of length~$d$. We prove an asymptotic formula conjectured by these authors, and we give explicitly the multiplicative constant appearing in the result, answering a question they asked.
These two results should make the OEIS richer by 100+25=125 dollars.

 In the case $r=1$ we recover Regev's result for permutations. Our proof goes as follows: expressing $A_{d,r}(n)$ as a sum over tableaux via the RSK correspondence, we show that the only tableaux contributing to the sum are ``almost'' rectangular (in the scale $\sqrt{n}$). This relies on asymptotic estimates for the Kotska numbers $K_{\lambda,r^n}$ when $\lambda$ has a fixed number of parts. Contrarily to the case $r=1$ where these numbers are given by the hook-length formula, we don't have closed form expressions here, so to get our asymptotic estimates we rely on more delicate computations, via the Jacobi-Trudi identity and saddle-point estimates.
\end{abstract}
\maketitle

\vspace{-5mm}

\section{Introduction}

Fix integers $r\geq 1$, $d\geq 3$. For $n \geq 1$, we let $A_{d,r}(n)$ be the number of words of length $nr$ over the alphabet $\{1,2,\dots,n\}^*$, such that each letter $i\in \{1,2,\dots,n\}$ appears exactly $r$ times, and having no increasing subsequence of length $d$.
This quantity was introduced in \cite{EZ}. Our main result is the following theorem, conjectured in~\cite{EZ} (for aesthetic reasons we prefer to state the result for $A_{d+1,r}(n)$ rather than for $A_{d,r}(n)$).
\begin{theorem}\label{thm:main}
For fixed $d\geq 2$ and $r\geq 1$, the number of words of length $nr$ over $\{1,2,\dots,n\}^*$ in which each letter appears exactly $r$ times, and having no increasing subsequence of length $d+1$, satisfies when $n$ tends to infinity:
$$
A_{d+1,r}(n) \sim
C_{d+1,r}  
\cdot 
n^{-\frac{(d+1)(d-1)}{2}}
\left[d^r {r+d-1 \choose d-1}\right]^n
$$
with 
\begin{align*}
C_{d+1,r}
&=
\frac{\sqrt{d}\cdot \prod_{i=1}^{d-1} i!}{(2\pi)^{\frac{d}{2}-\frac{1}{2}}}
\left(\frac{d(d+1)}{r(2d+r+1)}\right)^{\frac{(d-1)(d+1)}{2}}.
\end{align*}
\end{theorem}
This formula was conjectured in~\cite{EZ}, apart from the explicit value of the constant $C_{d+1,r}$ that was only conjectured for small values of $d$ and $r$.

Our first step in the proof of Theorem~\ref{thm:main} is the following remark, already made by the authors of \cite{EZ}. By the RSK correspondence (see for example~\cite{EC2}), we can express $A_{d+1,r}(n)$ as a sum over partitions of $rn$ having at most $d$ parts
\footnote{note that this equation is slightly wrong in \cite{EZ} (the inequality sign isn't correct).}
\begin{eqnarray}\label{eq:RSK}
A_{d+1,r}(n) =  \sum_{\lambda \vdash rn \atop \ell(\lambda)\leq d} f_\lambda g_{\lambda}^{(r)},
\end{eqnarray}
where $f_{\lambda}$ is the number of standard Young tableaux of shape $\lambda$, and where $g_{\lambda}^{(r)}$ is the number of semi-standard Young tableaux of shape $\lambda$ having content $1^r 2^r \dots n^r$. In other words, $g_{\lambda}^{(r)}$ is the Kotska number $K_{\lambda, r^n}$.

Note that in the case $r=1$, one has $f_\lambda=g_{\lambda}^{(r)}$. In this case, Regev~\cite{Regev} (see also \cite{M,Novak}) analysed the sum~\eqref{eq:RSK} by studying precisely the asymptotics of the numbers $f_\lambda$. In particular, he showed that for $r=1$, the sum \eqref{eq:RSK} is dominated by tableaux whose shape is ``close'' to the rectangular shape $(n/d, n/d, \dots, n/d)$, up to deviations of order $\sqrt{n}$. In the rest of this paper we are going to apply the same program for general~$r$.

\section{Estimates for the Kostka numbers $K_{\lambda, r^n}$}

We refer to \cite{EC2} for background on symmetric functions and for the terminology and notation we use here.
Let $\lambda=(\lambda_1,\lambda_2, \dots, \lambda_d)$ be a partition of $rn$ having at most $d$ parts (we complete $\lambda$ with zeros if $\ell(\lambda)<d)$. The Schur function $s_\lambda\equiv s_\lambda(x_1,x_2,\dots,x_n)$ can be expressed via the Jacobi-Trudi identity:
\begin{align}\label{eq:JT}
s_\lambda = \det \big( h_{\lambda_i+j-i} \big)_{1\leq i,j \leq d}
\end{align}
where $h_k$ is zero if $k<0$, and $h_k$ is the \emph{complete symmetric} function of the $x_i$ otherwise, \textit{i.e.}:
$$
h_k(x_1, x_2,\dots, x_n) = \sum_{1\leq i_1\leq i_2 \leq \dots\leq i_k \leq n} x_{i_1}x_{i_2}\dots x_{i_k}.
$$
Note that \eqref{eq:JT} is a determinant of \emph{fixed} sized $d$, even when $n$ tends to infinity (this will be crucial for us). Now recall that the Kostka number $K_{\lambda, r^n} = g_\lambda^{(r)}$ is equal to the coefficient of the monomial symmetric function $m_{1^r 2^r \dots n^r}$ in $s_\lambda$:
\begin{align}\label{eq:coeff}
g_{\lambda}^{(r)} = [x_1^r x_2^r \dots, x_n^r] s_\lambda(x_1, x_2, \dots, x_n).
\end{align}
Combining this equation with \eqref{eq:JT}, we will obtain what will be the starting point of our asymptotic estimates:
\begin{lemma}\label{lemma:Kot}
The Kostka number $g_{\lambda}^{(r)}=K_{\lambda,r^n}$ is given by the following formula:
\begin{align*}
g_\lambda^{(r)}= 
\frac{1}{(2i\pi)^{d}}\oint \oint \dots \oint  \frac{\big(h_r (x_1,x_2,\dots,x_{d})\big)^n}{x_1^{\lambda_1} x_2^{\lambda_2}\dots x_{d}^{\lambda_{d}}}
V(x_1,x_2,\dots, x_{d}) \frac{dx_1}{x_1} \frac{dx_2}{x_2} \dots \frac{dx_{d}}{x_d}
\end{align*}
where the $d$ contour integrals are taken along contours encircling $0$, and where the function $V$ is defined by:
$$
V(x_1,x_2,\dots, x_d)=\prod_{1\leq i < j \leq d}(1-\frac{x_j}{x_i}).
$$
\end{lemma}
\noindent Note that the ``nice property'' in the last lemma is that the number of integrals is $d$, and does not depend on $n$. So the last formula is maybe not as nice as, say, a hook-length formula, but it is not too bad. The ``less nice property'' is that because of the Vandermonde, the integrand is not a series with positive coefficients (which will make the asymptotics more delicate).
\begin{proof}
This lemma may be already known, and we don't know if it has already been used elsewhere. We are going to prove it with the Jacobi Trudi formula.

First, combining the Jacobi-Trudi formula~\eqref{eq:JT} and Equation~\eqref{eq:coeff}, we have:
\begin{align}\label{eq:int1}
g_\lambda^{(r)} &= [x_1^rx_2^r\dots x_n^r] \det \big( h_{\lambda_i+j-i} \big)_{1\leq i,j \leq d}\\
 &= \sum_{\sigma \in \mathfrak{S}_d} \epsilon(\sigma) [x_1^rx_2^r\dots x_n^r] 
\prod_{i=1}^d h_{\lambda_i+\sigma_i-i} (x_1, x_2, \dots, x_n).\nonumber
\end{align}
Now, for integers $k_1, k_2, \dots, k_d$ of total sum $rn$, we have:
$$ 
[x_1^rx_2^r\dots x_n^r] \prod_{i=1}^d h_{k_i} (x_1,x_2,\dots, x_n)  
= [x_1^{k_1}x_2^{k_2}\dots x_d^{k_d}] \big(h_r(x_1,x_2,\dots, x_d)\big)^n.
$$
Indeed it is easy to see combinatorially that both sides count the number of ways of distibuting $r$ undistinguishable copies of each letter in $\{1,2,\dots n\}$ in $d$ urns of  respective sizes $k_1, k_2, \dots, k_d$. 
%Finally note that since $h_r$ is homogeneous of degree $r$ we have
%\begin{align*}
%[x_1^{k_1}x_2^{k_2}\dots x_d^{k_d}] \big(h_r(x_1,x_2,\dots, x_d)\big)^n
%&=[x_1^{k_1}x_2^{k_2}\dots x_{d-1}^{k_{d-1}}] \big(h_r(x_1,x_2,\dots, x_{d-1},1)\big)^n.
%\end{align*}
%Returning to~\eqref{eq:int1}, and 
Using Cauchy's integral formula, we have:
\begin{align*}
g_\lambda^{(r)} & = \sum_{\sigma \in \mathfrak{S}_d} \epsilon(\sigma) [x_1^{\lambda_1+\sigma_1-1}x_2^{\lambda_2+\sigma_2-2}\dots x_{d}^{\lambda_{d}+\sigma_{d}-d}] \big(h_r(x_1,x_2,\dots, x_{d})\big)^n\\
& = \sum_{\sigma \in \mathfrak{S}_d} \epsilon(\sigma) 
\frac{1}{(2i\pi)^{d}}\oint\oint\dots\oint \frac{\big(h_r(x_1,x_2,\dots, x_{d})\big)^n}{x_1^{\lambda_1+\sigma_1}x_2^{\lambda_2+\sigma_2-1}\dots x_{d}^{\lambda_{d}+\sigma_{d}+1-d}}
dx_1dx_2\dots dx_{d}\\
& = 
\frac{1}{(2i\pi)^{d}}\oint\oint\dots\oint \frac{\big(h_r(x_1,x_2,\dots, x_{d})\big)^n}{x_1^{\lambda_1}x_2^{\lambda_2-1}\dots x_{d}^{\lambda_{d}-d+1}}
\left(
\sum_{\sigma \in \mathfrak{S}_d} \epsilon(\sigma) 
\prod_{i=1}^{d} x_{i}^{-\sigma_i}
\right)dx_1dx_2\dots dx_{d}\\
%& = 
%\oint\oint\dots\oint \frac{\big(h_r(x_1,x_2,\dots, x_{d-1},1)\big)^n}{x_1^{\lambda_1}x_2^{\lambda_2-1}\dots x_{d-1}^{\lambda_{d-1}-d+2}}
%\det \left( x_{i}^{-j}
%\right)_{1\leq i,j\leq d} dx_1dx_2\dots dx_{d-1}.
\end{align*}
The lemma follows by evaluating the Vandermonde determinant.
\end{proof}
We now note that, since $h_r$ is homogeneous, we can reduce the integral to a $(d-1)$-dimensional integral. An obvious way to do that would be to write:
\begin{align*}
g_\lambda^{(r)}= 
\frac{1}{(2i\pi)^{d-1}}\oint \oint \dots \oint  \frac{\big(h_r (x_1,x_2,\dots,x_{d-1},1)\big)^n}{x_1^{\lambda_1} x_2^{\lambda_2}\dots x_{d-1}^{\lambda_{d-1}}}
V(x_1,x_2,\dots, x_{d-1},1) \frac{dx_1}{x_1}\frac{dx_2}{x_2} \dots \frac{dx_{d-1}}{x_{d-1}}.
\end{align*}
However we prefer not to break the (anti)symmetry of the formula, for reasons that will be clear later, so we will proceed as follows.
Let $(x_1, x_2, \dots, x_d) \in \mathbb{R}_{>0}^d$. We have from the previous lemma:
\begin{align*}
g_\lambda^{(r)}= 
\frac{1}{(2\pi)^{d}}\int d\theta \frac{\big(h_r (x_1e^{i\theta_1},x_2e^{i\theta_2},\dots,x_{d}e^{i\theta_d})\big)^n}{x_1^{\lambda_1} x_2^{\lambda_2}\dots x_{d}^{\lambda_{d}} e^{i \sum_{j=1}^d \lambda_j \theta_j}}
\prod_{j<k}\left(1-\frac{x_j}{x_k} e^{i(\theta_j -\theta_k)}\right)
\end{align*}
where the integral is taken over $[-\pi,\pi]^{d}$.
Now we make the change of variables 
$$\nu_0 = \frac{d-1}{d} \sum_{j=1}^d \theta_j \; ; \; 
\nu_k =\theta_k -\frac{1}{d}\sum_{j=1}^d \theta_j \mbox{ for } k<d,
$$
whose inverse is given by $\theta_k=\frac{1}{d-1}\nu_0+\nu_k$ if $k\leq d$, where we use the notation $\nu_d:=-(\nu_1+\nu_2+\dots+\nu_{d-1})$. Using that $h_r$ is homogeneous of degree $r$, the last integral is equal to: 
\begin{align*}
g_\lambda^{(r)}= 
\frac{1/(d-1)}{(2\pi)^{d}}\int d\nu \frac{\big(h_r (x_1e^{i\nu_1},x_2e^{i\nu_2},\dots, x_{d}e^{i\nu_d})\big)^ne^{irn\nu_0/(d-1)}}{x_1^{\lambda_1} x_2^{\lambda_2}\dots x_{d}^{\lambda_{d}} e^{i \sum_{j=1}^d \lambda_j (\nu_j+\frac{1}{d-1}\nu_0)}}
\prod_{1\leq j<k \leq d}\left(1-\frac{x_j}{x_k} e^{i(\nu_j -\nu_k)}\right)
\end{align*}
where the integral is over $(\nu_0, \nu_1, \dots, \nu_{d-1})$
in $[-(d-1)\pi,(d-1)\pi]\times [-\pi,\pi]^{d-1}$, and where the factor $1/(d-1)$ comes from the Jacobian of the change of variables.
We observe that, since $\lambda$ is a partition of $rn$, the integrand is independent of $\nu_0$, so that we can integrate with respect to $\nu_0$. We get, changing back the name of our variables from ``$\nu$'' to~``$\theta$'':
\begin{align*}%\label{eq:intReady}
g_\lambda^{(r)}= 
\frac{1}{(2\pi)^{d-1}}\int_\Lambda d\theta \frac{\big(h_r (x_1e^{i\theta_1},x_2e^{i\theta_2},\dots, x_{d}e^{i\theta_d})\big)^n}{(x_1e^{i\theta_1})^{\lambda_1} (x_2e^{i\theta_2})^{\lambda_2}\dots(x_{d}e^{i\theta_d})^{\lambda_{d}} }
\prod_{1\leq j<k \leq d}\left(1-\frac{x_je^{i\theta_j}}{x_ke^{i\theta_k}}\right)
\end{align*}
where the $(d-1)$-dimensional integral is over $(\theta_1, \dots, \theta_{d})$ under the projection $d\theta$ of the Lebesgue measure on the ``hyperplane'' 
$$\Lambda:=\{\theta\in[-\pi,\pi]^d, \; \theta_1+\theta_2+\dots+\theta_{d}=0 [2\pi]\}.$$
% and where we use the notation $\theta_d=-(\theta_1+\theta_2+\dots+\theta_{d-1})$.
More precisely $d\theta := d\theta_1 d\theta_2\dots d\theta_{d-1}$, although this way of writing it hides the fact that it is fully symmetric in the $d$ coordinates.

Now, note that in the last integral, because $h_r$ is homogeneous, the integrand remains unchanged if we simultaneously shift all $\theta_i$ by $\frac{2k\pi}{d}$ for some integer $0\leq k < d$ (and such a shift stabilizes $\Lambda$). 
We can thus decompose $\Lambda$ into $d$ subsets according to the class of $\theta_1$ modulo $\frac{2\pi}{d}$, and these $d$ subsets give the same contribution. We finally obtain:
\begin{corollary}
The Kostka number $g_{\lambda}^{(r)}=K_{\lambda,r^n}$ is given by the following formula:
\begin{align}\label{eq:intReady}
g_\lambda^{(r)}= 
\frac{d}{(2\pi)^{d-1}}\int_\Theta d\theta \frac{\big(h_r (x_1e^{i\theta_1},x_2e^{i\theta_2},\dots, x_{d}e^{i\theta_d})\big)^n}{(x_1e^{i\theta_1})^{\lambda_1} (x_2e^{i\theta_2})^{\lambda_2}\dots(x_{d}e^{i\theta_d})^{\lambda_{d}} }
\prod_{1\leq j<k \leq d}\left(1-\frac{x_je^{i\theta_j}}{x_ke^{i\theta_k}}\right)
\end{align}
with
$$\Theta:=\left\{\theta\in[-\pi,\pi)^{d}, \; \theta_1+\theta_2+\dots+\theta_{d}=0[2\pi]\,,\, -\frac{\pi}{d}\leq\theta_1<\frac{\pi}{d}\right\}
$$
and $d\theta=d\theta_1 d\theta_2\dots d\theta_{d-1}$ is the (scaled) Lebesgue measure on $\Theta$.
\end{corollary}

\noindent Now we start the asymptotic work. We are going to use the last formula and the saddle point method to obtain  precise asymptotic estimates of $g_\lambda^{(r)}$ when $\lambda$ is ``close'' or ``moderately close'' to be a rectangular shape -- large deviation estimates saying that $g_\lambda^{(r)}$ is small (in comparison) if $\lambda$ is ``far'' from a rectangular shape will be obtained later by more elementary means.

\subsection{Saddle-point estimates}

\newcommand{\xx}{\mathbf{x}}
\newcommand{\yy}{\mathbf{y}}
\newcommand{\zz}{\mathbf{z}}
\newcommand{\CC}{\mathbb{R}}
\newcommand{\CCS}{\mathbb{R}_0}
\newcommand{\CCP}{\mathbb{R}_\perp}

%As we remarked above, we have chosen to work with $d$ integrals whereas the integral given by the previous lemma is ``intrisically'' $(d-1)$-dimensional. So we will have to be careful about this slightly singular situation.
We decompose $\CC^d$ as the direct sum $\CC^d= \CCS \oplus \CCP$, where $\CCS$ is the set of vectors of sum zero, and $\CCP$ is the complex line generated by $(1,1,\dots,1)$.
In $\CC^d$ we will also consider the affine hyperplane $\CC_r$ consisting of vectors whose coordinates add up to $r$:
$$
\CC_r:=\{(y_1,y_2,\dots,y_d)\in\CC^d\;:\;y_1+y_2+\dots+y_d=r\}.
$$
Let us fix $\yy=(y_1,y_2,\dots, y_d) \in \CC_r$.
%Let us fix $(y_1, y_2, \dots, y_{d-1})\in \mathbb{R}_{>0}^{d-1}$, such that $\sum_{i=1}^{d-1} y_i <r$.
%We let $\lambda^y:=(y_1 \cdot n, y_2 \cdot n, \dots, y_d \cdot n )$ and we assume that $y$ is such that $\lambda$ is an integer partition of $rn$. 
For $\xx=(x_1,x_2,\dots, x_{d})\in\CC^d$ we define 
$$g(\xx) = \frac{h_r(\xx)}{x_1^{y_1}x_2^{y_2} \dots x_{d}^{y_{d}}}.$$
%Lemma \ref{lemma:contour} says that 
%We will be interested in the $d$-dimensional integral:
%$$
%I(\yy) =\frac{1}{(2i\pi)^{d}} \oint\oint\dots \oint
%V(\xx) g(\xx)^n d\xx,
%$$
%with $V(\xx)$ is defined as in the previous lemma.	
%Note that Lemma~\ref{lemma:contour} says that $g_\lambda^{(r)} = I(y)$ where $y_i=\lambda_i/n$ for $i\in [1..d]$, which is why we are interested in the quantity $I(y)$.

%It is natural to study this integral using the saddle-point method.
In order to study the last integral, we introduce the \emph{saddle-point} system, in the ``unknown'' $\xx^0=(x^0_1,x^0_2, \dots, x^0_{d})$:
\begin{align}\label{eq:saddle}
\left.\frac{\partial}{\partial x_i} g(\xx)\right|_{\xx=\xx^0}=0, \mbox{ for } 1\leq i \leq d.
\end{align}
We note by an explicit computation that \eqref{eq:saddle} is equivalent to:
\begin{align}\label{eq:saddle2}
y_j h_r(\xx^0) = x^0_j \left.\frac{\partial h_r(\xx)}{\partial x_j}\right|_{\xx=\xx^0}, \mbox{ for } 1\leq i \leq d, 
\end{align}
%where we use the notation $h_r(\xx)=h_r(x_1,x_2,\dots, x_{d-1},1)$.
Note that since $h_r$ is homogeneous of degree $r$, we have:
$$
\sum_{j=1}^d x_j \frac{\partial h_r(\xx)}{\partial x_j}=r h_r(\xx),
$$
hence the assumption we made, that $\yy\in\CC_r$, is necessary for this system to have a solution.

\begin{lemma}
For $\yy = (r/d, r/d, \dots, r/d)$, the point $\xx^0 = (1,1,\dots,1)$ is a solution of the saddle-point equations~\eqref{eq:saddle2}.
Moreover, there exists a vicinity $V$ of $(r/d, r/d, \dots, r/d)$ in $\CC_r$ such that for $\yy \in V$, the saddle-point system $\eqref{eq:saddle2}$ has a unique solution $\xx^0=\xx^0(\yy)$. Moreover, the mapping $\yy \mapsto \xx^0(\yy)$ is analytic on $V$.
\end{lemma}
\begin{proof}
First, for $\xx=(1,1,\dots, 1)$ one has $h_r(\xx)={r+d-1 \choose d-1}$ and $x_j \frac{\partial h_r(\xx)}{\partial x_j} = {r+d-1 \choose r-1}$.
Since ${r+d-1 \choose d-1} = \frac{d}{r}  {r+d-1 \choose r-1}$, this shows that for $\yy = (r/d, r/d, \dots, r/d)$, the point $\xx=(1, 1, \dots, 1)$ is indeed a solution of~\eqref{eq:saddle2}.

To prove the remaining assertion, we use the implicit function theorem (or more precisely the constant rank theorem since here we are actually in a $(d-1)$-dimensional situation inside a $d$-dimensional space). 
Consider the mapping  $\phi: U \rightarrow \CC^d$ given by $\xx \stackrel{\phi}{\mapsto} (z_1, z_2, \dots, z_{d})$ with $z_j:= \frac{x_j}{h_r(\xx)}\frac{\partial h_r(\xx)}{\partial x_j}$, where $U$ is a neighborhood of $(1,1,\dots,1)$ in $\CC^{d}$. Note that $\phi(U)\subset \CC_r$. Moreover, the saddle-point equations~\eqref{eq:saddle2} say that $\yy = \phi(\xx^0)$. 
%Therefore what we want to prove is that $\phi$ has an inverse (analytic) function in a vicinity of $(1, 1, \dots, 1)$.
Now the Jacobian matrix of $\phi$ is given by:
$$
J(\xx)  = \left(\frac{\partial z_j}{\partial x_i}\right)_{1\leq i,j\leq d-1}
\mbox{ where } 
\frac{\partial z_j}{\partial x_i}=
\left\{
\begin{array}{ll}
\frac{(h_r^{j}+x_j h_r^{jj})h_r-x_j (h_r^j)^2}{(h_r)^2},& i=j,\\ 
\frac{(x_j h_r^{ij})h_r-x_j h_r^i h_r^j}{(h_r)^2},& i\neq j, 
\end{array}
\right.
$$
where functions are evaluated at $\xx$ and exponents of $h_r$ indicate partial derivatives, for example $h_r^{j}=\frac{\partial}{\partial x_j} h_r (\xx)$.
At the point $\xx=(1, 1, \dots, 1)$, we have from routine calculations that $h_r^i = \frac{r}{d} h_r$ for each $i$, and moreover $h_r^{ii} = \frac{2r(r-1)}{d(d+1)} h_r$ and $h_r^{ij} = \frac{r(r-1)}{d(d+1)} h_r$ for $i\neq j$. It follows that:
\begin{align}\label{eq:jacob}
\frac{\partial z_j}{\partial x_i}(1,1,\dots, 1)=
\left\{
\begin{array}{ll}
\alpha:=\frac{r(d-1)(d+r)}{d^2(d+1)},& i=j,\\ 
\beta:=-\frac{r(d+r)}{d^2(d+1)},& i\neq j. 
\end{array}
\right.
\end{align}
Therefore $J(1,1,\dots,1) = (\alpha -\beta) id + \beta N$ where $N$ is the all-one matrix of size~$d$. Therefore the eigenvalues of $J(1,1,\dots, 1)$ are $\alpha-\beta\neq 0$ ($d-1$ times) and $\alpha+(d-1)\beta=0$ (once). The two corresponding eigenspaces are respectively $\CCS$ 
%(or more precisely $\CC_r$ if we want to think about it as a tangent space in the ambient $\CC^d$) 
and $\CCP$.
From the constant rank theorem, it follows 
% and they are all nonzero, which shows that $J(1,1,\dots,1)$ is invertible.
%It follows that the Jacobian matrix $J(\xx)$ is inversible for $\xx$ in a neighbourhood of $(1, 1, \dots, 1)$ and has an analytical inverse. 
%From the implicit function theorem, this proves 
that the mapping $\phi$ is a locally analytic, invertible map, between a neighborhood of $(1,1,\dots,1)$ in a $(d-1)$-dimensional subvariety of of $\CC^d$ containing $(1,1,\dots,1)$, and a neighbourhood of $(r/d, r/d,  \dots, r/d)$ in $\CC_r$. This concludes the proof.
\end{proof}

If $\zz=(z_1,z_2,\dots, z_d)$ is an element of $\mathbb{R}^d$, we let $\mathrm{gap}(\zz) = \min_{1\leq i< j\leq d} |z_i-z_j|$.
\begin{proposition}\label{prop:uniform}
Fix $\gamma>0$, and 
consider the set 
$L_{\gamma}\subset \mathbb{R}^{d}$ defined by:
$$L_{\gamma} = 
\left\{
\begin{array}{c}
(\zz\in\mathbb{R}^d: \; z_1+z_2+\dots+z_d =0\;;\; 
\mathrm{gap}\big(\zz\big) >\gamma \;; 
 \; \max_{1\leq i \leq d} |z_i|<n^{1/8}
\end{array}
\right\}.$$
Let $\zz \in L_{\gamma}$ and let
$\lambda=(\lambda_1, \lambda_2,\dots, \lambda_d)$ with $\lambda_i = \frac{rn}{d}+z_i \sqrt{n}$. Assume that $\zz$ is such that $\lambda$ is an integer partition (of $rn$).
%
%
%There exists a neighborhood $W\subset V$ of $(r/d,r/d, \dots,r/d)$ in $\mathbb{R}^{d-1}$ such that, for $y=(y_1,y_2,\dots, y_{d-1})\in W$, with $y_1>y_2>\dots>y_{d-1}$ we have:
Then when $n$ tends to infinity, we have:
$$
g_\lambda^{(r)} = \frac{\sqrt{d}\alpha'^{\frac{(d+1)(d-1)}{2}}{d+r-1 \choose d-1}^nn^{-\frac{(d+2)(d-1)}{4}}}{(2\pi)^{\frac{d-1}{2}}}e^{-\frac{1}{2}Q(\zz)} \prod_{1\leq i<j \leq d}(z_j-z_i)  (1+o(1))
$$
where $Q(\zz):=\alpha'\sum_{i=1}^d z_i^2$, where $\alpha'=\frac{d(d+1)}{r(d+r)}$, and where the little-o is \emph{uniform} for $\zz\in L_{\gamma}$.
\end{proposition}
\begin{proof}
In the following proof, we say that a constant is \emph{uniform} if for each $\gamma$ that constant can be chosen uniformly for $\zz \in L_{\gamma}$.

We first let $\yy=(y_1,y_2,\dots, y_{d})$ with $y_i= \frac{r}{d} + \frac{z_i}{\sqrt{n}}$, so that $n\cdot y_i = \lambda_i$. We note that $\yy$ converges, uniformly, to $(r/d,r/d,\dots, r/d)$, so by the previous lemma for $n$ large enough there is a unique $\xx^0$ solution of $\eqref{eq:saddle2}$.
We can then choose the coordinates of $\xx^0$ as the radii of integration in formula~\eqref{eq:intReady}, to express $g_\lambda^{(r)}$ as:
$$
g_\lambda^{(r)}= I(y):=
\frac{d}{(2\pi)^{d-1}}\int_\Theta d\theta 
V(\xx^0\otimes e^{i\theta}) g(\xx^0\otimes e^{i\theta})^n
$$
where we use the notation $\xx^0\otimes e^{i\theta}:=(x_1^0 e^{i\theta_1}, x_1^0 e^{i\theta_1},x_2^0 e^{i\theta_2},\dots, x_{d}^0 e^{i\theta_{d}})$, and where as before $\Theta$ is the $(d-1)$-dimensional subspace of $[-\pi,\pi]^d$ formed by vector of null-sum (modulo $2\pi$). $\Theta$ can be viewed as a subspace of  $\CC_0$, that contains $(0, 0, \dots, 0)$ in its interior.

According to usual saddle-point heuristic we are going to partition  $\Theta$ into three subsets $A$, $B$, $C$, on which the contributions  will be studied separately. We use the notation $I(y)=I_A(y)+I_B(y)+I_C(y)$ to distinguish the three contributions to the integral of the three contours $A,B,C$. 
%Most of the wkr will take place in a neighboorhood of $\Theta=0$, so we first need to study the local expansions of $\xx$ at this point.

\noindent{$\bullet$ \textit{Computation of second derivatives.}}
We first evaluate the entries of the Hessian matrix from the explicit expression of $g(\xx)=h_r(\xx)/\prod_ix_i^{y_i}$.
Let $g_{i,j}=\left.\left(\frac{x_ix_j}{g(\xx)}\frac{\partial^2}{\partial x_i\partial x_j} g(\xx)\right) \right|_{\xx=\xx^0}$,
 where $\xx^0$ is related to $\yy$ by~\eqref{eq:saddle2}.
 We find:
$$
\left.\left(\frac{x_ix_j}{g(\xx)}\frac{\partial^2}{\partial x_i\partial x_j} g(\xx)\right)\right|_{\xx=\xx^0}
=
\left\{
\begin{array}{ll}
 \left(y_i (1-y_i) h_r +x_i^2 h_r^{ii} \right) /h_r& \mbox{ if } i= j,\\
\left(x_ix_j h_r^{ij} -y_iy_j h_r \right) /h_r & \mbox{ if } i\neq j, 
\end{array}
\right.
$$
with the same notation as in the previous proof.
% In particular $G=(g_{i,j})_{1\leq i,j\leq d-1}$ depends analytically of $\xx$, for $\yy$ in a neighborhood of $(r/d, r/d, \dots, r/d)\in \mathbb{R}^{d}$. 
Moreover, for $\yy=(r/d, r/d, \dots, r/d)$ and $\xx=(1,1,\dots, 1)$ we find the explicit expressions:
$$
g_{i,j} = \alpha  \mbox{ if } i=j \; ; \; g_{i,j} =  \beta  \mbox{ if } i\neq j,
$$
with $\alpha, \beta$ as above.
It follows that the eigenvalues of $G$ at $\yy=(r/d,r/d,\dots, r/d)$ 
are $(\alpha-\beta)\neq 0$ ($d-1$ times) and $(\alpha+(d-1)\beta)=0$ (once).
The eigenspace for the nonzero eigenvalue is $\CCS$, and for the zero eigenvalue it is $\CCP$.
%, and 
% and they are all nonzero and positive. I
%it follows that the eigenvalues of $G$ are nonzero and positive for $\yy$ in a small enough neighbourhood of $(r/d,r/d,\dots, r/d)$ contained in $V$ (call this neighborhood $W$).

%\noindent{$\bullet$ \textit{Contours of integration.}}
%We now fix $\yy$ in $W$, and we let $\xx^0=(x^0_1,x^0_2,\dots, x^0_{d})$ be the solution of~\eqref{eq:saddle2} granted by the previous lemma.
% We will evaluate the integral $I(y)$ on the saddle-point torus: 
%$$T:=\{ \xx^0\otimes e^{i\theta}, \theta\in(-\pi,\pi]^{d}\},$$
%where we use the notation $\xx^0\otimes e^{i\theta}:=(x_1^0 e^{i\theta_1}, x_1^0 e^{i\theta_1},x_2^0 e^{i\theta_2},\dots, x_{d}^0 e^{i\theta_{d}})$.

\noindent{$\bullet$ \textit{Local expansion of $\xx^0$.}}
First note that when $n$ tends to infinity, with the hypotheses of the proposition, $\yy$ tends to $(r/n, r/n, \dots, r/n)$ uniformly in the sense above. More precisely we have $|\frac{z_i}{\sqrt{n}}|\leq n^{1/8-1/2}=n^{-3/8}$.
Recall moreover that the Jacobian matrix $J(1,1,\dots,1)$ given by $\eqref{eq:jacob}$ acts as the scalar $\alpha-\beta=\frac{r(d+r)}{d(d+1)}$ on its eigenspace $\CCS$.
Therefore we have:
$$\xx^0 = (1,1,\dots,1) + n^{-1/2} \alpha'\zz + O(n^{-\frac{3}{4}}),
$$
 where $\alpha'=\frac{d(d+1)}{r(d+r)}$, and where the big-O is uniform.

\noindent{$\bullet$ \textit{First contour of integration (``large arguments'').}}
For $\epsilon>0$ we let $\Theta_\epsilon \subset \Theta$ be defined by
$$\Theta_\epsilon =\textstyle
 \left\{\theta \in \Theta:\; \forall i\in [1..d], \;|\theta_i|<\epsilon\right\}.
$$
%In other words, $T_\epsilon$ is the subset of $T$ on which the $\CCS$-component of the angle vector~$\theta$ is smaller than $\epsilon$, in $L^1$ norm.
% $T_\epsilon:=\{ \xx^0\otimes e^{i\theta}, \theta\in[-\epsilon,\epsilon]^{d}\}$.
%From the local expansion of $g$ at $\xx^0$,
From the definition of $g$ and from the fact that $\xx^0\in\mathbb{R}_{>0}^{d}$ 
there exists $\epsilon>0$ such that for all $\xx \in \Theta\setminus \Theta_\epsilon$, one has $|g(\xx)| \leq c'<g(\xx^0)$ for some constant $c'$ uniform for $\yy$ in a small enough neighbourhood  $(r/d,r/d,\dots, r/d)$ in $\CC_r$
(to see that, by the strict triangle inequality,  it is enough to show that if $\theta\in \Theta$ is nonzero, at least one monomial in the expansion of $h_r(\xx^0\otimes e^{i\theta})$ has a non-zero argument; but for the contrary to be true, all the $e^{i\theta_j}$ have to be equal to a common $r$-th root  of unity; in particular the $\theta_j$ have to be all equal, and since  they must sum to $0$ mod $2\pi$, we must have $\theta_j=\frac{2k\pi}{d}$ for some $k$; but from the restriction on $\theta_1$ in the definition of $\Theta$ this forces $k=0$).
We pick such an~$\epsilon$, and we define our first contour of integration we as $A:=\Theta\setminus \Theta_\epsilon$.
Now let $\yy$ be as in the statement of the proposition. For $n$ large enough, uniformly for $\zz \in L_{\gamma}$, $\yy$ is in that neighborhood. Moreover, $|V(\xx)|$ is 
bounded, so we obtain:
%at most polynomial in $n$.
%, and from the local expansion of $\xx^0$ above, we have 
%$|V(\xx^0)|> c'' n^{-\frac{1}{2} d(d-1)/2}$ for some uniform constant $c''$ (this is where we use the hypothesis on the gap). 
%Hence $V(\xx^0)$ is at most polynomially smaller than $V(\xx)$ for $\xx\in A$. This polynomial factor can be absorbed by changing slightly the exponential constants, and we find that for $\xx\in A$ we have:
$$
|V(\xx) g(\xx)^n| < c_2 (c_4)^n g(\xx^0)^n
$$
for uniform constants $c_2$ and $c_4<1$. It follows that 
\begin{align}\label{eq:IA}
I_A(y) < c_2 (c_4)^n g(\xx^0)^n.
\end{align}

\noindent{$\bullet$ \textit{Second contour of integration (``moderate arguments'').}}
We now consider the contour $C:=\Theta_{n^{-2/5}}$ with the same notation as above, and we define our second contour of integration as $B:=\Theta_\epsilon\setminus C$. The contour $B$ corresponds to points of $T$ whose arguments  are ``moderate'' (smaller than~$\epsilon$, but not all smaller than $n^{-2/5}$)\footnote{Classically in saddle-point methods, the choice of the exponent $-2/5$ is motivated by the fact that $(n^{-2/5})^2 \cdot n\rightarrow \infty$ but $(n^{-2/5})^3 \cdot n\rightarrow 0$.}. Recalling that the partial derivatives vanish at the saddle point $\xx^0$, we have for $\xx \in T_\epsilon$:
\begin{align}\label{eq:devg2}
g(\xx) = g(\xx^0) \left(1- \frac{1}{2}\sum_{1\leq i, j \leq d} g_{i,j} \theta_i\theta_j 
+ O(\epsilon^3)\right)
\end{align}
where the big-O is uniform for $\yy$ in a neighborhood of $(r/d, r/d, \dots, r/d)$. Moreover, up to reducing, if necessary, the value of $\epsilon$ picked in the previous step (which does not modify its conclusions), the maximum of the function $|g(\xx)|$ on $B$ is at most equal to $ g(\xx^0) (1-c' n^{-4/5})$ for some uniform constant $c'>0$.
% $|g(\xx^0\otimes e^{i(n^{-2/5}, \dots, n^{2/5})}|$. Therefore, using the local expansion~\eqref{eq:devg2}, we find that, for $x\in B$, we have:
$$
|g(\xx)|^n \leq c\cdot g(\xx^0)^n (1- c' n^{-4/5})^n \leq c\cdot g(\xx^0)^n \exp(-c'n^{1/5}),
$$ 
for uniform constants $c, c'>0$. 
Noting as before that $|V(\xx)|$ is bounded %at most polynomially large, 
%>c'n^{-\frac{1}{2}d(d-1)/2}$ is at most polynomially small, 
%we can absorb this polynomial correction by slightly changing the exponential constant and 
we find that
\begin{align}\label{eq:IB}
|I_B(y)| \leq c''g(\xx^0)^n\exp(-c_3 n^{1/5})
\end{align}
 where again all constants are uniform.

\noindent{$\bullet$ \textit{The dominating contour (``small arguments'').}}
We now look at the contour $C$, which as we will see is the only one to contribute to the result.
First, for $\xx\otimes e^{i\theta} \in C$, we have the local expansion, recalling that the first partial derivatives vanish at $\xx^0$:
$$
g(\xx^0\otimes e^{i\theta}) = g(\xx^0) \left(1- \frac{1}{2} \sum_{i,j} g_{i,j} \theta_i \theta_j+  O(n^{-6/5})\right).
$$
Now, since the entries of the matrix $G=(g_{i,j})_{1\leq i,j<d}$ depend analytically on $\xx^0$, we have:
$$
G = G_1 + O(n^{-3/8}),
$$
where $G_1:=1/\alpha' id + \beta N$, is the value of $G$ for $\xx^0=(1,1,\dots,1)$. 
It follows that 
$
g(\xx^0\otimes e^{i\theta}) = g(\xx^0) \left(1- \theta G_1 \theta^T + O(n^{-6/5})\right).
$
Note that, since $\theta \in \Theta$, we have $\theta N =0 $ and therefore
$
g(\xx^0\otimes e^{i\theta}) = g(\xx^0) \left(1- \alpha'^{-1}\theta \theta^T + O(n^{-6/5})\right).
$

Similarly, from the local expansion of $\xx^0$ we have
$g(\xx^0)=g(1,1,\dots,1) (1- \frac{1}{2n} \alpha' \zz \zz^T + O(n^{-9/8}))$, 
% where $J_1=(\alpha -\beta) id + \beta N$ is the Jacobian matrix at $(1,1,\dots, 1)$ computed before.
%We note that $J_1=G_1$, and since $\sum_i z_i=0$ we have $\zz N = 0$, and it 
\textit{i.e.} $g(\xx^0) =g(1,1,\dots,1) (1- \frac{1}{2n} Q(\zz) + O(n^{-9/8}))$ with $Q(\zz)$ as in the statement of the proposition.
Therefore, we have\footnote{Here we see the explanation for the choice of the exponent $n^{1/8}$ for the bound on $\max_i |z_i|$. More generally, if we had imposed $\max_i |z_i| < n^{\kappa}$, the term $O(n^{-9/8})$ would have become $O(n^{3(\kappa-\frac{1}{2})})$, and we see that for this term to be a $o(n^{-1})$, we need $\kappa<\frac{1}{6}$ (so $\frac{1}{8}$ was a safe choice). This critical exponent $\frac{1}{6}=\frac{2}{3}-\frac{1}{2}$ is classical in ``moderate deviations'' results, where ``central limit'' estimates are often valid up to any region of width $o(n^{2/3})$, and not only $O(n^{\frac{1}{2}})$.}
$$
g(\xx^0)^n = e^{-\frac{1}{2}Q(\zz)} (1+o(1)),
$$
with a \emph{uniform} little-o.
Recalling that $g(1,1,\dots,1)={d+r-1\choose d-1}$, the contribution of the contour $C$ to the integral can thus be rewritten:
\begin{align*}
I_C(y)&=\frac{d \cdot {d+r-1 \choose d-1}^ne^{-\frac{1}{2}Q(\zz)}}{(2\pi)^{d-1}} 
\int_{\theta_1, \theta_2, \dots, \theta_d \atop {|\theta_i|<n^{-2/5}\atop \theta_1+\theta_2+\dots +\theta_d=0}}
V(\xx^0 \otimes e^{i\theta}) 
\left( 1 - 1/\alpha'\theta\theta^T
+ O(n^{-6/5})\right)^n (1+o(1))d\theta.
\end{align*}
Therefore, making the change of variables $\theta_j = \nu_i /\sqrt{n}$, the last integral is equal to:
\begin{align}\label{eq:IC2}
\frac{d\cdot n^{-\frac{d-1}{2}}{d+r-1 \choose d-1}^ne^{-\frac{1}{2}Q(\zz)}}{(2\pi)^{d-1}} 
\int_{\nu_1, \nu_2, \dots, \nu_d \atop {|\nu_i|<n^{1/10}\atop \nu_1+\nu_2+\dots +\nu_d=0}}
V(\xx^0 \otimes e^{i\nu/\sqrt{n}}) 
\left( 1 - \frac{1}{\alpha'n}\nu\nu^T
+ O(n^{-6/5})\right)^n (1+o(1))d\nu.
\end{align}
Now note that we have the uniform estimate 
$$
\left( 1 - \frac{1}{\alpha'n} \nu\nu^T
+ O(n^{-6/5})\right)^n = e^{-1/\alpha' \nu\nu^T} (1+o(1)).
$$
We are going to inject this estimate in the integral (and justify the validity of this operation in the next paragraph). The quantity~\eqref{eq:IC2} becomes, up to a $(1+o(1))$ factor:
\begin{align*}
\frac{d\cdot n^{-\frac{d-1}{2}}{d+r-1 \choose d-1}^ne^{-\frac{1}{2}Q(\zz)}}{(2\pi)^{d-1}} 
\int_{\nu_1, \nu_2, \dots, \nu_d \atop {|\nu_i|<n^{1/10}\atop \nu_1+\nu_2+\dots +\nu_d=0}}
V(\xx^0 \otimes e^{i\nu/\sqrt{n}}) 
e^{-1/\alpha' \nu\nu^T} d\nu, 
\end{align*}
which up to exponentially small terms, is equal to:
\begin{align*}
\frac{d\cdot n^{-\frac{d-1}{2}}{d+r-1 \choose d-1}^ne^{-\frac{1}{2}Q(\zz)}}{(2\pi)^{d-1}} 
\int_{\nu_1, \nu_2, \dots, \nu_d \atop {\nu_1+\nu_2+\dots +\nu_d=0}}
V(\xx^0 \otimes e^{i\nu/\sqrt{n}}) 
e^{-1/\alpha' \nu\nu^T} d\nu.
\end{align*}
If we view $(x^0_1,x^0_2,\dots, x^0_d)$ as \emph{formal} variables, then the last quantity is antisymmetric in the $x_i$ (since the Gaussian measure in the exponential is symmetric), therefore it is a scalar multiple of the Vandermonde determinant $\prod_{1\leq j < k\leq d}(x^0_k-x^0_j)$. Moreover the multiplicative constant is easily determined by looking at the coefficient of $(x^0_1)^{d-1}$, and, by integrating the Gaussian density, it is equal to $\frac{1}{\sqrt{d}}(\sqrt{2\pi})^{d-1}\alpha'^{\frac{d-1}{2}}(1+o(1))$, where the $\frac{1}{\sqrt{d}}$ comes from the projection of the Lebesgue measure from $\mathbb{R}^{d-1}$ to $\CCS$.  
Therefore the last quantity is equal to:
\begin{align*}
\frac{\sqrt{d}n^{-\frac{d-1}{2}}{d+r-1 \choose d-1}^ne^{-\frac{1}{2}Q(\zz)}\alpha'^{\frac{d-1}{2}}}{(2\pi)^{\frac{d-1}{2}}} 
\prod_{1\leq j < k\leq d}(x^0_k-x^0_j) (1+o(1))
\end{align*}
which from the local expansion of $\xx^0$ is equal, up to a uniform $(1+o(1))$ factor, to:
\begin{align}\label{eq:toBeDiscussed}
\frac{\sqrt{d}n^{-\frac{d-1}{2}-\frac{d(d-1)}{4}}{d+r-1 \choose d-1}^ne^{-\frac{1}{2}Q(\zz)}\alpha'^{\frac{d-1}{2}+\frac{d(d-1)}{2}}}{(2\pi)^{\frac{d-1}{2}}} 
\prod_{1\leq j < k\leq d}(z^0_k-z^0_j),
\end{align}
which is the result announced in the statement of the proposition. 
%Since the integrals $I_A(y)$ and $I_B(y)$ are exponentially smaller, as are the tails of the Gaussian integrals, it only 
It remains to justify that the approximations we have made 
%just after writing \eqref{eq:IC2}
 were correct.
%the approximation of $g(\xx^0 \otimes e^{i\theta})$ by $e^{-\frac{1}{2}Q(\zz)}$ in \eqref{eq:IC2} leads to a correct asymptotics for $I_C(y)$.
%It remains to justify that the approximation we have made was correct.
Consider the integral in Formula \eqref{eq:IC2}.
This formula was obtained from the definition of $I_C(y)$ by replacing $g(\xx^0\otimes e^{i\theta})^n$ by its first order approximation inside the integral. It is not obvious that the final result of this operation gives an equivalent of $I_C(y)$, since in the course of the computation, the Vandermonde $\prod_{j<k} (x^0_k-x^0_j)$ contributed to a (very small) factor of $O(n^{-d(d-1)/4})$ to \eqref{eq:toBeDiscussed}. The problem is that, \textit{a priori}, a small correction term in the expansion of $g(\xx^0\otimes e^{i\theta})^n$ could lead to a higher contribution in the end, if this term was not driven down by a Vandermonde or some other small factor.
However, let us prove now that this is not the case. 
The idea is that the antisymmetry property that we have used is valid, in an approximate sense, to all orders.

\newcommand{\dy}{\mathbf{\delta{y}}}

More precisely, note that $g(\xx^0\otimes e^{i\theta})$, viewed as a function of $\theta$ for $\xx^0$ fixed
%and $\theta$ (with $\xx^0$ linked to $\yy$ by \eqref{eq:saddle2}) 
is analytic at $\theta=0$.
%More precisely, since $g(\xx)$ is analytic at $\xx=(1,1,\dots,1)$,
By composition, the same is true for $\ln \frac{g(\xx^0\otimes e^{i\theta})}{g(\xx^0)}$.
Therefore, since $\theta\in C$ converge at controlled speed to $0$
(namely, at speed $n^{-2/5}$) we can write a Taylor expansion of $\ln g(\xx^0\otimes e^{i\theta})$ with a remainder term which is as small as we want. 
That is, for any $L\geq 1$, we can get an expansion of the form:
%We can then expand $g(\xx^0\otimes e^{i\theta})^n$ and get
$$
\ln \frac{g(\xx^0\otimes e^{i\theta})}{g(\xx^0)} = 
 -\theta G\theta^T + \sum_{k\geq 3\atop \frac{2}{5}k<L} P_{k}(\theta,\xx^0) +O(n^{-L})
$$
where $P_k(\theta,x)$ is a homogeneous polynomial of degree $k$ in $\theta$, whose coordinates depend analytically on $\xx^0$, where $G=G(\xx^0)$ is the Hessian, and where the constant in the big-O is absolute (\textit{i.e.} independant of $\xx^0$ in a neighborhood of $(1,1,\dots, 1)$).
Note that we have used that first derivatives vanish at $\xx^0$, by~\eqref{eq:saddle2}.
Thererefore, letting $\nu=\theta\sqrt{n}$, we have 
$$
 n \cdot \ln \frac{g(\xx^0\otimes e^{i\theta})}{g(\xx^0)} = 
 -\nu G \nu^T + \sum_{k\geq 3\atop \frac{2}{5}k\leq L} n^{1-\frac{k}{2}} P_{k}(\nu,\xx^0) +O(n^{-L-1}).
$$
We can thus substitute this in the definition of $I_C(y)$ and get:
\begin{align*}
I_C(y) &=  d\cdot g(\xx^0)^n \cdot \frac{n^{-\frac{d-1}{2}}}{(2\pi)^{d-1}} 
\int_{\nu_1, \nu_2, \dots, \nu_d \atop {|\nu_i|<n^{1/10}\atop \nu_1+\nu_2+\dots +\nu_d=0}}
V(\xx^0\otimes e^{i\nu/\sqrt{n}})
e^{\left(
-\nu G \nu^T + \sum\limits_{k\geq 3\atop \frac{2}{5}k<L} n^{1-\frac{k}{2}} P_{k}(\nu,\xx^0) +O(n^{-L})
\right)}\\
&=\frac{d\cdot n^{-\frac{d-1}{2}}}{(2\pi)^{d-1}} g(\xx^0)^n\cdot
\int_{\nu_1, \nu_2, \dots, \nu_d \atop {|\nu_i|<n^{1/10}\atop \nu_1+\nu_2+\dots +\nu_d=0}}
%V(\xx^0\otimes e^{i\nu/\sqrt{n}})
e^{-\nu G \nu^T} \left( \sum_{0\leq k<2L} n^{-k/2}Q_{k}(\nu,\xx^0) +O(n^{-L}) \right)
\end{align*}
for some polynomials $Q_k$ in $\nu$
% that are homogeneous of degree $k$ in $\nu$ and 
whose entries depend analytically on $\xx^0$ in a neighbourhood of $(1,1,\dots,1)$, and where in the second equality we have also replaced the Vandermonde by its Taylor expansion in $\theta$. We can thus complete the Gaussian tails and get:
\begin{align*}
I_C(y) &= g(\xx^0)^n E(n) +\\ &\frac{d\cdot n^{-\frac{d-1}{2}}}{(2\pi)^{d-1}} g(\xx^0)^n\cdot
\int_{\nu_1, \nu_2, \dots, \nu_d \atop { \nu_1+\nu_2+\dots +\nu_d=0}}
e^{-\nu G \nu^T} \left( \sum_{0\leq k<2L} n^{-\frac{k}{2}} Q_{k}(\nu,\xx^0) +O(n^{-L}) \right)
\end{align*}
where $E(n)$ is an exponentially decreasing term. Integrating the Gaussian density term by term (and recalling that the big-O is absolute), we obtain:
\begin{align*}
I_C(y) &= g(\xx^0)^n E(n) +\frac{\sqrt{d}\cdot n^{-\frac{d-1}{2}}}{(2\pi)^{d-1}} g(\xx^0)^n\cdot \sqrt{\det G}
\left( \sum_{0\leq k<2L} n^{-\frac{k}{2}} R_{k}(\xx^0) +O(n^{-L}) \right)
\end{align*}
for quantities $R_k(\xx^0)$ that depend analytically on $\xx^0$.
Now, the quantities $g(\xx^0)$ and $\sqrt{\det G}$ also depend analytically on  $\xx^0$ near $(1,1,\dots,1)$.
 We can thus replace each of these quantities by a Taylor expansion up to a large order, and we can ensure that the error term is at most $O(n^{-L})$, uniformly.
We prefer to write this Taylor expansion in terms of $\yy$ tending to $(r/d,r/d,\dots r/d)$. We thus obtain an expression of the form:
\begin{align}
I(y) &= {r+d-1\choose d-1} ^n E(n) +\frac{n^{-\frac{d-1}{2}}}{(2\pi)^{d-1}} {r+d-1 \choose d-1}^n e^{-\frac{1}{2}Q(\zz)}\left( \sum_{0\leq k<2L} n^{-\frac{k}{2}} S_{k}(\dy) +O(n^{-L}) \right),\label{eq:lastExpr}
\end{align}
where $\dy=(y_1-\frac{r}{d}, y_2-\frac{r}{d}, \dots, y_d-\frac{r}{d})$, where the $S_k(\dy)$ are polynomials in $\dy$,  and where the exponentially decreasing term $E(n)$ has been, if necessary, modified, to take into account the contributions of $I_A(y)$ and $I_B(y)$.

Now comes the crucial observation. From its original definition, we note  that the integral $I(y)$ is antisymmetric in  the coordinates of $\dy':=(\delta y_1-\frac{1}{n},\delta y_2-\frac{2}{n},\dots, \delta y_d-\frac{d}{n})$. We can transform \eqref{eq:lastExpr} into a similar-looking expansion in terms of $\dy'$ instead of $\dy$, up to replacing the polynomials $S_k$ by other polynomials $S_k'$ (we just substitute $y_i=y'_i+i/n$, and we collect the powers of $n$). Since $E(n)$ is exponentially small, the antisymmetry implies that each \emph{polynomial} $S'_k(\dy')$ is divisible by the Vandermonde:
$$
\prod_{1\leq i < j \leq d} (\dy'_i-\dy'_j).
$$
This proves that the term $n^{-\frac{k}{2}} S_k'(\dy')$ in the sum is a 
$$O\left( \prod_{1\leq j<k \leq d}(z_k-z_j) n^{-d(d-1)/4-k/2}\right).$$ Now, tracking the computations, we see that the term corresponding to $k=0$, that consists in selecting the first order term in all the successive expansions, is precisely the term that we have considered when deriving~\eqref{eq:toBeDiscussed}, and gives a contribution of $\prod_{1\leq j<k \leq d}(z_k-z_j) n^{-d(d-1)/4}$ times a constant factor. Therefore it is strictly larger than all other terms in the expansion, and since for $\zz \in L_\gamma$ the product $\prod_{1\leq j<k \leq d}(z_k-z_j)$ is bounded away from $0$, it is also strictly larger than the remainder $O(n^{-L})$ provided we chose $L$ large enough. We thereby obtain a full proof that $\eqref{eq:toBeDiscussed}$ is indeed a valid approximation of $I(y)$.
\end{proof}

\begin{remark}\label{rem:smallGaps}
If we do not assume that $\zz \in L_{\gamma}$, the previous proposition fails because we can't control uniformly the behaviour of the Vandermonde $V(\xx^0)$ at the saddle-point. However, the proof of the previous proposition also shows the following.
Let $\zz$ such that $\sum_i z_i=0$ and $\max_i |z_i|<n^{1/8}$, 
and let $\lambda=(\lambda_1, \lambda_2,\dots, \lambda_d)$ with $\lambda_i = \frac{rn}{d}+z_i \sqrt{n}$. Assume that $\zz$ is such that $\lambda$ is an integer partition (of $rn$).
%and let $\yy=(y_1,y_2,\dots, y_{d-1})$ with $y_i= \frac{r}{d} + \frac{z_i}{\sqrt{n}}$
Then when $n$ tends to infinity, we have with the notation of the previous proposition:
$$
\left| g_\lambda^{(r)} - I_C(y) 
\right|
 \leq 
 c \exp(-c'n^{1/5}){r+d-1 \choose d-1}^n, 
$$
with $\displaystyle |I_C(y)|\leq c {d+r-1 \choose d-1}^n n^{-\frac{(d+2)(d-1)}{4}}e^{-\frac{1}{2}Q(\zz)}  P(|\zz|) $, for uniform constants $c, c'$ and for some polynomial $P$.
\end{remark}
\begin{proof}
We split the integral $I(y)=I_A(y)+I_B(y)+I_C(y)$ as in the proof of the previous proposition. The fact that when $n$ tends to infinity $y_i$ tends to $(r/n,r/n,\dots, r/n)$ is still valid \emph{uniformly}, so all the local estimates we used in the proof are still uniformly valid.
Now, the estimates of $I_A(y)$ and $I_B(y)$ remain valid. As for $I_C(y)$, the last argument of the proof does not work, since the Vandermonde can be arbitrarily small.  
However, the same proof shows that we can get an upper bound on $I_C(y)$ by replacing the quantity $\prod_{1\leq j\leq k\leq d} (z_k-z_j)$ in the result by $\prod_{1\leq j\leq k\leq d} (z_k-z_j+O(\frac{1}{\sqrt{n}}))$. The modulus of this quantity is clearly polynomially bounded in $\zz$.
\end{proof}

\section{Contribution to the sum}

We can now analyse the wanted sum. Recall that we have 
$$
A_{d+1,r}(n) 
= \sum_{\lambda \vdash rn \atop \ell(\lambda)\leq d} f_\lambda g_\lambda^{(r)}
= \sum_{\lambda \vdash rn \atop \ell(\lambda)\leq d} g_\lambda^{(1)} g_\lambda^{(r)}.
$$
We split this sum into several parts. Let $\Lambda$ be the set of partitions of $rn$ into at most $d$ parts. We fix a small parameter $\gamma>0$ we let:
$$
\Lambda_1 = \{\lambda \in \Lambda: \;  |\lambda_i-\frac{rn}{d}|\leq n^{1/2+1/8} \;\mbox{ and }\; \forall i\neq j, |\lambda_i-\lambda_j| \geq \gamma \sqrt{n}\}
$$ 
$$
\Lambda_2 = \{\lambda \in \Lambda: \;  \exists i: |\lambda_i-\frac{rn}{d}| >n^{1/2+1/8} \}
$$
$$
\Lambda_3 = \{\lambda \in \Lambda: \;  |\lambda_i-\frac{rn}{d}|<n^{1/2+1/8} \; \exists i\neq j, |\lambda_i-\lambda_j| < \gamma\sqrt{n}\}
$$ 
Note that $\Lambda=\Lambda_1\uplus \Lambda_2 \uplus\Lambda_3$.

\medskip

\noindent $\bullet$ \textit{Contribution of $\Lambda_2$}.
We first analyse the sum over $\Lambda_2$. Let $\lambda\in \Lambda_2$. The number of semi-standard Young Tableaux of shape $\lambda$ and content $r^n$ (resp $1^{rn}$) is at most the number of ways of placing $r$ (resp $1$) instiguishable balls of label $i$ for each $i=1..n$ (resp. $i=1..rn$) into urns of sizes $\lambda_1, \lambda_2, \dots, \lambda_d$. This number is ${r+d-1 \choose d-1}^n$ (resp $d^{rn}$) times the probability that a uniform random assignement of the balls into $d$ urns ends up with urns of size $\lambda_1, \lambda_2, \dots, \lambda_d$. By Chernoff's bound, for $\lambda \in \Lambda_2$, this probability is smaller than $\exp(-c n^{1-2/8})$ for some $c>0$.
Since the number of elements in $\Lambda$ is only polynomial in $n$, we deduce that the contribution of the set $\Lambda_2$ to the sum is bounded as follows:
$$
\sum_{\lambda \vdash rn \atop \lambda\in\Lambda_2} f_\lambda g_\lambda^{(r)} \leq c' \left[d^{r} {r+d-1 \choose d-1}\right]^n \exp(-c'' n^{3/4})
$$
for constants $c', c''>0$.

\noindent $\bullet$ \textit{Contribution of $\Lambda_1$ and $\Lambda_3$}.
For $\lambda$ in $\Lambda_1$, we can apply Proposition~\ref{prop:uniform} to get an estimate of the number $g_\lambda^{(r)}$ (and apply it with $r=1$ and replacing $n$ by $rn$ to get an estimate for $f_\lambda$). We find:
\begin{align*}	
f_\lambda g_\lambda^{(r)}
= d \left(\frac{d^2(d+1)}{r(d+r)}\right)^{\frac{(d+1)(d-1)}{2}}
\frac{1}{(2\pi)^{d-1}}
\prod_{1<i<j\leq d}(z_i-z_j)^2 e^{-R(\zz)}\\
\times n^{-\frac{(d+2)(d-1)}{2}}
r^{-\frac{(d-1)(d+1)}{2}}
\left[d^r {r+d-1 \choose d-1}\right]^n
(1+o(1))
\end{align*}
where $\zz=(z_1,z_2,\dots, z_{d})$ with $\lambda_i = \frac{rn}{d}+\sqrt{n} z_i$, $R(\zz)=\left(\frac{d}{r}+\frac{d(d+1)}{r(d+r)}\right)\sum_{i=1}^d z_i^2$.
 Note the power of $r$, that accounts from the fact that to apply Proposition~\ref{prop:uniform} with $rn$ instead of $n$, we have to scale the $z_i$ by $1/\sqrt{r}$ (hence a factor of $r^{-\frac{(d+2)(d-1)}{4}}$ coming from the power of $n$, and a factor of $r^{-\frac{d(d-1)}{4}}$ coming from the Vandermonde, which gives $r^{-\frac{(d-1)(d+1)}{2}}$ in total ).
%and where $c(\zz)$ tends to $0$ uniformly for $\max_i |z_i|$ on compact sets (by Corollary~\ref{cor:bounded}). Moreover note that $c(\zz)$ is \emph{uniformly} bounded by Corollary~\ref{cor:unbounded}.
Therefore 
\begin{align*}	
f_\lambda g_\lambda^{(r)}
\left( d \left(\frac{d^2(d+1)}{r(d+r)}\right)^{\frac{(d+1)(d-1)}{2}}
\frac{n^{-\frac{(d+2)(d-1)}{2}}
r^{-\frac{(d-1)(d+1)}{2}}
}{(2\pi)^{d-1}}
\left[d^r {r+d-1 \choose d-1}\right]^n
\right)^{-1}\\ 
\rightarrow\prod_{1<i<j\leq d}(z_i-z_j)^2 e^{-R(\zz)}.
\end{align*}
Moreover, by the uniformity of the little-o in Proposition~\ref{prop:uniform}, this convergence is dominated. 

On the other hand, for $\lambda \in \Lambda_3$, we can use Remark~\ref{rem:smallGaps} to get:
\begin{align*}	
f_\lambda g_\lambda^{(r)} 
\left( d \left(\frac{d^2(d+1)}{r(d+r)}\right)^{\frac{(d+1)(d-1)}{2}}
\frac{n^{-\frac{(d+2)(d-1)}{2}}}{(2\pi)^{d-1}}
\left[d^r {r+d-1 \choose d-1}\right]^n
\right)^{-1} 
\leq P(\zz) e^{-R(\zz)} + E,
\end{align*}
where $E$ is an exponentially small term, that we can thus disregard.

Since this is true for any $\gamma>0$ small enough, we easily deduce from the  the dominated convergence theorem\footnote{Apply the dominated convergence theorem for $\gamma$ fixed to see that the sum we want to evaluate is close to the integral in the right-hand-side up to a $O(\gamma)$ factor, and then let $\gamma$ tend to zero.} that:
\begin{align*}
\left(
\frac{d}{(2\pi)^{d-1}}
 \left(\frac{d^2(d+1)}{r(d+r)}\right)^{\frac{(d+1)(d-1)}{2}}
n^{-\frac{(d+2)(d-1)}{2}}
\left[d^r {r+d-1 \choose d-1}\right]^n
r^{-\frac{(d-1)(d+1)}{2}}
\right)^{-1} \\
\times
n^{\frac{1-d}{2}}\sum_{\lambda\in \Lambda_1\uplus\Lambda_3}f_\lambda g_\lambda^{(r)} 
\rightarrow
%O(\gamma) + 
\int_{\nabla}
\prod_{1\leq i<j \leq d}
(z_i -z_j)^2 \exp \left(- \frac{d(2d+r+1)}{2r(d+r)}\sum_{i=1}^d z_i^2\right) 
\end{align*}
where the $(d-1)$-dimensional integral is taken over the set $\nabla:=\{z_1+z_2+\dots+z_d=0, \; z_1<z_2<\dots<z_d\}$.
% minus the exceptionnal set $S_\gamma:=\{\exists i,j, |z_i-z_j|\leq \gamma\}$.
%We can now let $\gamma$ tend to zero and obtain:
%\begin{align*}
%\left(
%\frac{d}{(2\pi)^{d-1}}
% \left(\frac{d^2(d+1)}{r(d+r)}\right)^{\frac{(d+1)(d-1)}{2}}
%n^{-\frac{(d+2)(d-1)}{2}}
%\left[d^r {r+d-1 \choose d-1}\right]^n
%r^{-\frac{(d-1)(d+1)}{2}}
%\right)^{-1}\\
%\times n^{\frac{1-d}{2}}
%\sum_{\lambda\in \Lambda_1\uplus\Lambda_3}f_\lambda g_\lambda^{(r)}
%\rightarrow
%\int_{\nabla}
%\prod_{1\leq i<j \leq d}
%(z_i -z_j)^2 \exp \left(- \frac{d(2d+r+1)}{2r(d+r)}\sum_{i=1}^d z_i^2\right) 
%\end{align*}
This integral was evaluated in Regev's original paper~\cite{Regev}, who showed that:
$$
\int_{\nabla}
\prod_{1\leq i<j \leq d}
(z_i -z_j)^2 \exp \left(-\sum_{i=1}^d z_i^2\right)
=\frac{1}{d!}\sqrt{\frac{1}{\pi d}}
(2\pi)^{\frac{d}{2}}2^{-\frac{d^2}{2}} \prod_{i=1}^d i!.
$$
Making the change of variable $z_i =\sqrt{\frac{d(2d+r+1)}{2r(d+r)}}z'_i$ this enables us to evaluate the integral we need.
Putting this together with the previous estimate we finally have proved:
\begin{align*}
\sum_{\lambda\in \Lambda_1\uplus\Lambda_3}f_\lambda g_\lambda^{(r)} 
\sim
\frac{d}{(2\pi)^{d-1}}
 \left(\frac{d^2(d+1)}{r(d+r)}\right)^{\frac{(d+1)(d-1)}{2}}
n^{-\frac{(d+1)(d-1)}{2}}
\left[d^r {r+d-1 \choose d-1}\right]^n
r^{-\frac{(d-1)(d+1)}{2}}\\
\times \frac{1}{d!}\sqrt{\frac{1}{\pi d}}
(2\pi)^{\frac{d}{2}}2^{-\frac{d^2}{2}} \prod_{i=1}^d i!
\left(\frac{d(2d+r+1)}{2r(d+r)}\right)^{-\frac{d-1}{2}-\frac{d(d-1)}{2}}
\end{align*}
Therefore we finally have proved that when $n$ tends to infinity we have:
$$
A_{d+1,r}(n) \sim
C_{d+1,r}  
n^{-\frac{(d+1)(d-1)}{2}}
\left[d^r {r+d-1 \choose d-1}\right]^n
$$
with 
\begin{align*}
C_{d+1,r}
&=
\frac{\sqrt{d}\cdot \prod_{i=1}^{d-1} i!}{(2\pi)^{\frac{d}{2}-\frac{1}{2}}}
 \left(\frac{d^2(d+1)}{r(d+r)}\right)^{\frac{(d+1)(d-1)}{2}}
r^{-\frac{(d-1)(d+1)}{2}}
2^{-\frac{d^2}{2}} 
\left(\frac{d(2d+r+1)}{2r(d+r)}\right)^{-\frac{(d+1)(d-1)}{2}}\\
&=
\frac{\sqrt{d}\cdot \prod_{i=1}^{d-1} i!}{(2\pi)^{\frac{d}{2}-\frac{1}{2}}}
\left(\frac{d(d+1)}{r(2d+r+1)}\right)^{\frac{(d-1)(d+1)}{2}}
\end{align*}
which is what we wanted to prove!
\bibliographystyle{alpha}
\bibliography{biblio}

\begin{thebibliography}{Mat08}

\bibitem[EZ14]{EZ}
S.~B. {Ekhad} and D.~{Zeilberger}.
\newblock {The Generating Functions Enumerating 12..d-Avoiding Words with r
  occurrences of each of 1,2, ... , n are D-finite for all d and all r}.
\newblock {\em The Personal Journal of Shalosh B. Ekhad and Doron Zeilberger,
  Dec. 5, 2014.}, 2014.
\newblock See also http://arxiv.org/abs/1412.2035.

\bibitem[Mat08]{M}
Sho Matsumoto.
\newblock Jack deformations of {P}lancherel measures and traceless {G}aussian
  random matrices.
\newblock {\em Electron. J. Combin.}, 15(1):Research Paper 149, 18, 2008.

\bibitem[Nov11]{Novak}
Jonathan Novak.
\newblock An asymptotic version of a theorem of {K}nuth.
\newblock {\em Adv. in Appl. Math.}, 47(1):49--56, 2011.

\bibitem[Reg81]{Regev}
Amitai Regev.
\newblock Asymptotic values for degrees associated with strips of {Y}oung
  diagrams.
\newblock {\em Adv. in Math.}, 41(2):115--136, 1981.

\bibitem[Sta99]{EC2}
R.~P. Stanley.
\newblock {\em Enumerative combinatorics. {V}ol. 2}, volume~62 of {\em
  Cambridge Studies in Advanced Mathematics}.
\newblock Cambridge University Press, Cambridge, 1999.
\newblock With a foreword by Gian-Carlo Rota and appendix 1 by Sergey Fomin.

\end{thebibliography}

%Wie first note that $\xx=\yy$ is a saddle-point of the integral. 
%More precisely, for each $1\leq i\leq d$ we have:
%$$
%\left. \frac{\partial}{\partial x_i} g(\xx) \right|_{\xx=\yy} =0.
%$$

\end{document}